\documentclass[abstracton]{scrartcl} 

\usepackage{amsfonts}
\usepackage{amsmath}
\usepackage{amssymb}
\usepackage{xcolor}
\usepackage{dsfont}

\usepackage{graphics}
\usepackage{graphicx}
\usepackage{epstopdf}
\usepackage{epsf}

\usepackage{subcaption}
\usepackage{float}

\usepackage{latexsym}
\usepackage{oldgerm}
\usepackage{dsfont}
\usepackage{amscd, amsthm,enumerate,verbatim,calc}
\usepackage{authblk}
\usepackage{enumitem}

\newcommand{\N}{{\mathbb{N}}}

\newtheorem{theorem}{Theorem}
\newtheorem{Lemma}[theorem]{Lemma}
\newtheorem{proposition}[theorem]{Proposition}
\newtheorem{corollary}[theorem]{Corollary}

\newtheorem{fact}[theorem]{Fact}

\newcommand{\dist}{{\rm dist}}

\newcommand{\w}{\color{black}}

\usepackage[colorlinks,breaklinks,backref]{hyperref}
\usepackage{hyperref}
\usepackage{backref}
\usepackage{todonotes}

\begin{document}

\title {
On 3-colourability of $(bull,H)$-free graphs}

\author[1]{Nadzieja Hodur}
\author[1]{Monika Pil\'sniak}
\author[1]{Magdalena Prorok}
\author[2]{Ingo Schiermeyer}
\affil[1]{\normalsize AGH University of Krakow, al. Mickiewicza 30, 30-059 Krak\'ow, Poland}
\affil[2]{\normalsize TU Bergakademie Freiberg, 09596 Freiberg, Germany}

\date{\today}

\maketitle
\begin{abstract}
    The $3$-colourability problem is a well-known NP-complete problem and it remains NP-complete for $bull$-free graphs, where $bull$ is the graph consisting of $K_3$ with two pendant edges attached to two of its vertices. In this paper we study $3$-colourability of $(bull,H)$-free graphs for several graphs $H$. We show that these graphs are $3$-colourable or contain an induced odd wheel $W_{2p+1}$ for some $p\geq 2$ or a spindle graph $M_{3p+1}$ for some $p\geq 1$. Moreover, for all our results we can provide certifying algorithms that run in polynomial time.
\end{abstract}

\noindent
{\w Keywords: $3$-colourable graphs, forbidden induced subgraphs, perfect graphs, complexity\\               
Math. Subj. Class.: { 05C15,  05C17, 68Q25, 68W40.} }                      

\section{Introduction}\label{sec:intro}

We consider finite, simple, and  undirected graphs.
For terminology and notations not defined here, we refer to~\cite{BM08}.

An {\it induced subgraph} of a graph $G$
is a graph on a vertex set $S \subseteq V(G)$
for which two vertices are adjacent
if and only if they are adjacent in $G$.
In particular, we say that the subgraph is \emph{induced by $S$}.
We also say that a graph $H$ is an \emph{induced subgraph} of $G$ if $H$~is
isomorphic to an induced subgraph of $G$.

Given a family $\cal{H}$ of graphs and a graph $G$, we say that $G$ is \emph{$\cal{H}$-free}
if $G$ contains no graph from $\cal{H}$ as an induced subgraph.
In this context, the graphs of $\cal{H}$ are referred to as
\emph{forbidden induced subgraphs}.

A graph is {\it $k$-colourable} if each of its vertices can be coloured with one of $k$ colours
so that adjacent vertices obtain distinct colours.
The smallest integer $k$ such that a given graph $G$ is $k$-colourable
is called the {\it chromatic number} of $G$, denoted by $\chi(G).$
Clearly, $\chi(G) \geq \omega(G)$ for every graph $G$,
where $\omega(G)$ denotes the \emph{clique number} of $G$,
that is, the order of a maximum complete subgraph of $G$.
%
Furthermore, a graph $G$ is {\it perfect} if $\chi(G')=\omega(G')$
for every induced subgraph $G'$ of $G.$
For a subgraph $H$ and a vertex $v$, let $d_H(v) = |N(v) \cap V(H)|$.

The graph on five vertices $v_1$, $v_2$, $v_3$, $v_4$, $v_5$ and with the edges $v_1v_2$, $v_2v_3$, $v_3v_4$, $v_4v_5$, $v_2v_4$ is called the \emph{bull}. 
Let $S_{i,j,k}$ be the graph consisting of three induced paths of lengths $i$, $j$ and $k$, with a common initial vertex.
The graph $S_{1,1,1}$ is called \emph{claw}, $S_{1,1,2}$ is called \emph{chair} and $S_{1,2,2}$ is called $E$.

The $3$-colourability problem is a well-known NP-complete problem and it remains NP-complete for $claw$-free and $bull$-free graphs. In the last two decades, a large number of results of colourings of graphs with forbidden subgraphs have been shown (cf.~\cite{BDHKRSV22}, \cite{alpha3}, \cite{clawdiamondnet}, \cite{Ran}, \cite{RS04-1}, \cite{RST02}, \cite{Sum} and cf.~\cite{GJPS17}, \cite{RS04}, \cite{RS19} for three surveys).

Following \cite{BHS}, an algorithm is certifying, if it returns with each input a simple and easily verifiable certificate that the particular input is correct. For example, a certifying algorithm for the bipartite graph recognition would return either a $2$-colouring of the input graph proving that it is bipartite, or an odd cycle proving that it is not bipartite. In this paper we study $3$-colourability of $(bull, H)$-free graphs for several graphs $H$. For all of our results we will provide certifying algorithms that run in polynomial time. 

Our research has been motivated by~\cite{clawdiamondnet} and we use some definitions and notations from it. A graph $G$ of order $3p+1$, $p\geq1$ is called a \emph{spindle graph} $M_{3p+1}$ if it contains a~cycle $C$: $u_0 u_1 \ldots  u_{3p} u_0$, where $\{u_{3i-2}, u_{3i-1}, u_{3i+1}, u_{3i+2}\}=N_G(u_{3i})$ and $\{u_{3i-3}, u_{3i}\}=N_G(u_{3i-1}) \cap N_G(u_{3i-2})$ for each $i \in [p]$, where $[p]:=\{1, 2, \ldots, p\}$.

Observe that $M_4 \cong K_4$ and $M_7$ is known as the Moser spindle.

\begin{figure}[H]
    \centering
    \includegraphics[width=0.4\linewidth]{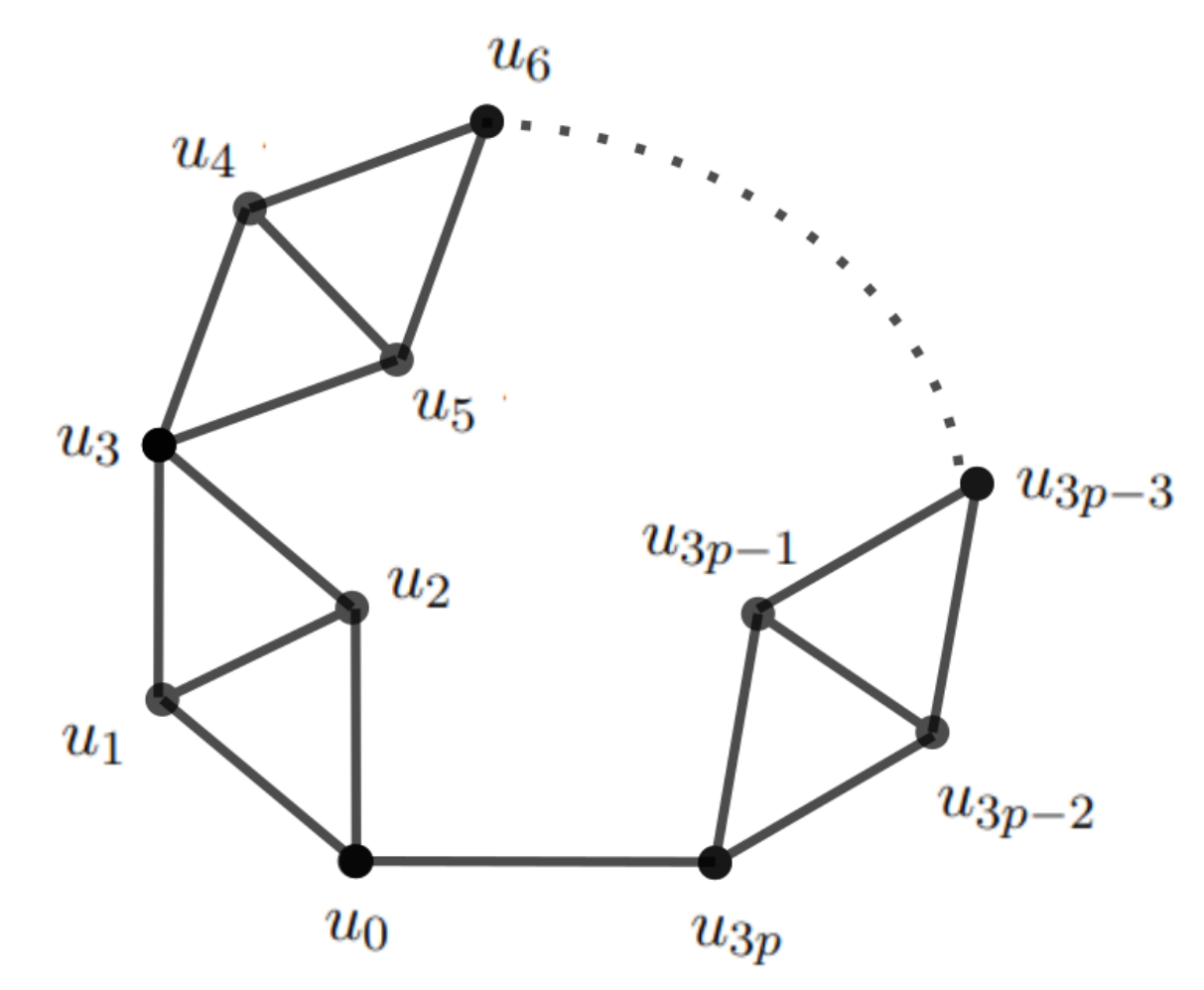}
    \label{spindler}
    \caption{The spindle graph $M_{3p+1}$.}
\end{figure}

\begin{proposition}[\rm \cite{clawdiamondnet}]
    The graph $M_{3p+1}$ is not 3-colourable for every $p \geq 1$. 
\end{proposition}

Since the $3$-colourability problem is NP-complete for claw-free graphs and $K_3$-free graphs (cf.~\cite{GJPS17}), it is also NP-complete for $bull$-free graphs.
The following theorem in~\cite{clawdiamondnet} has motivated our research.

\begin{theorem}[\rm \cite{clawdiamondnet}]
    Let $G$ be $(bull, claw)$-free graph. Then one of the following holds
    \begin{enumerate}[label=(\roman*)]
        \item $G$ contains $W_5$ or
        \item $G$ contains a (not necessarily induced) spindle graph $M_{3i+1}$ for some $i\geq1$ or 
        \item $G$ is 3-colourable.
    \end{enumerate}
\end{theorem}

The goal of this paper is to consider 3-colourability of $(bull, H)$-free graphs, where $H$ is a supergraph of the claw.

\begin{theorem}\label{thm:bullchair} 
Let $G$ be a connected $(bull, chair)$-free graph. Then 
\begin{enumerate}[label=(\roman*)]
    \item $G$ contains an odd wheel or
    \item $G$ contains a (not necessarily induced) spindle graph $M_{3i+1}$ for some $i\geq 1$ or
    \item $G$ is 3-colourable.
\end{enumerate}
\end{theorem}

In fact Theorem~\ref{thm:bullchair} can be extended to the larger class of $(bull,E)$-free graphs. However, for this proof, we will show and make use of several additional graph properties. 

\begin{theorem}\label{thm:bullE} 
Let $G$ be a connected $(bull, E)$-free graph. Then 
\begin{enumerate}[label=(\roman*)]
    \item $G$ contains an odd wheel or
    \item $G$ contains a (not necessarily induced) spindle graph $M_{3i+1}$ for some $i\geq 1$ or
    \item $G$ is 3-colourable.
\end{enumerate}
\end{theorem}

If we forbid in addition induced 5-cycles, then Theorem~\ref{thm:bullE} can be extended as follows.

\begin{theorem}\label{thm:bullH} 
Let $G$ be a connected $(bull, C_5, H)$-free graph with $H \in \{S_{1,1,3}, S_{1,2,3}\}$. Then 
\begin{enumerate}[label=(\roman*)]
    \item $G$ contains an odd wheel or
    \item $G$ contains a (not necessarily induced) spindle graph $M_{3i+1}$ for some $i\geq 1$ or
    \item $G$ is 3-colourable.
\end{enumerate}
\end{theorem}

The 3-colourability problem has been also studied for $P_k$-free graphs for $k\geq 5$.
Let $G_1$, $G_2$, $G_3$ be graphs on $7$, $10$ and $13$ vertices, respectively (see Figure~\ref{G123}).  In \cite{BHS} the following theorem was shown. 

\begin{figure}[htb]
    \begin{center}
    \begin{subfigure}{0.23\textwidth}
\includegraphics[width=0.9\linewidth]{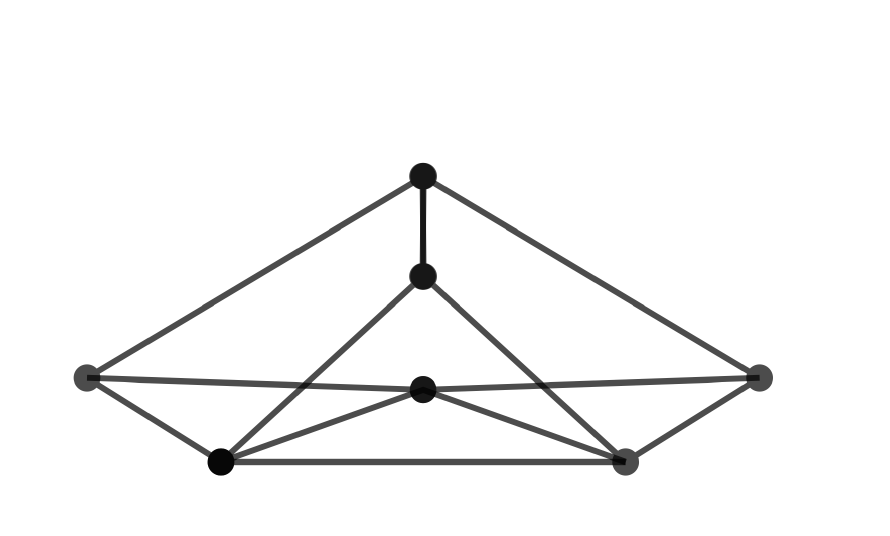} 
    \caption*{$G_1$}
    \end{subfigure}
    \begin{subfigure}{0.33\textwidth}\includegraphics[width=0.9\linewidth]{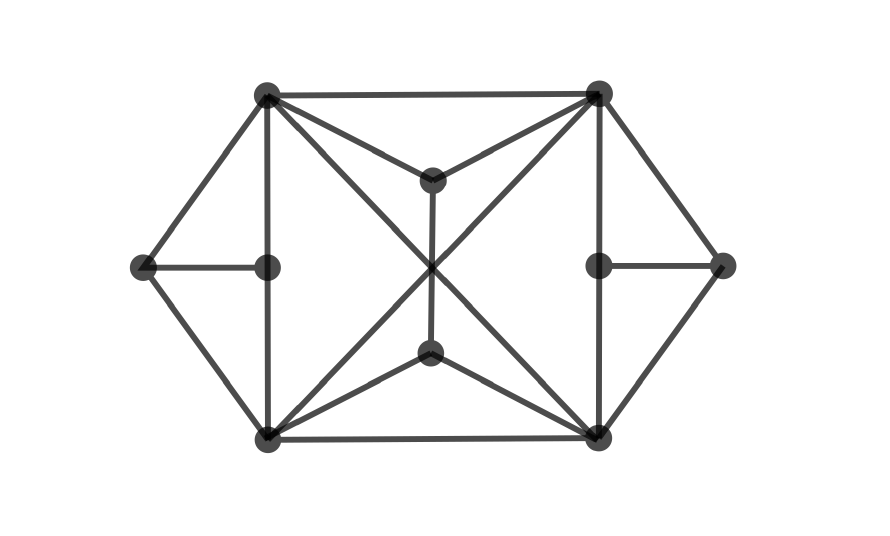}
    \caption*{$G_2$}
    \end{subfigure}
    \begin{subfigure}{0.33\textwidth}\includegraphics[width=0.9\linewidth, height=.5\linewidth]{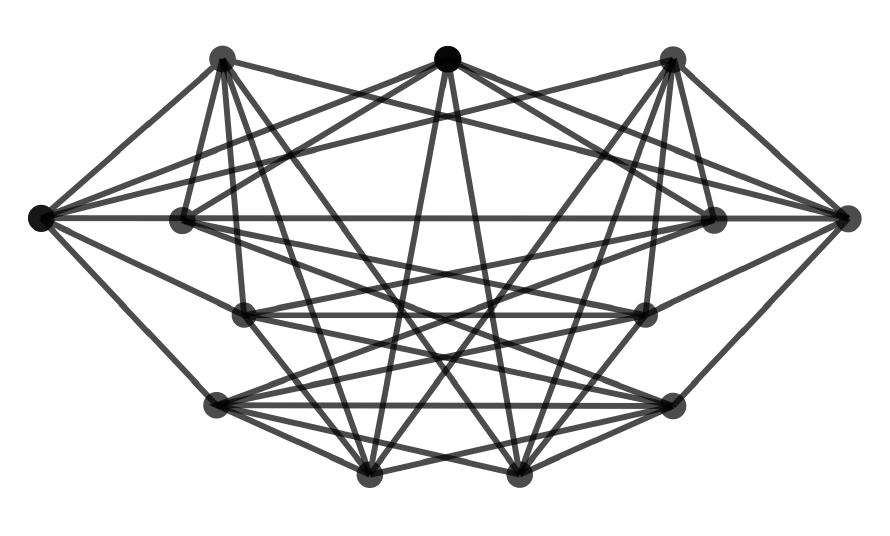}
    \caption*{$G_3$}
    \end{subfigure}
    \caption{The graphs $G_1$, $G_2$ and $G_3$ from Theorem \ref{BHS}.}
    \label{G123}
    \end{center}
\end{figure}

\begin{theorem}[\rm \cite{BHS}]\label{BHS}
    A $P_5$-free graph is $3$-colourable if and only if it does not contain $K_4$, $W_5$, $M_7$, $G_1$, $G_2$ or $G_3$ as a subgraph. 
\end{theorem}

Note that $G_1$, $G_2$ and $G_3$ are not $bull$-free. This leads to the following corollary from Theorem \ref{thm:bullH}. 

\begin{corollary}
    Let $G$ be a $(bull, P_5)$-free graph. Then
    \begin{enumerate}[label=(\roman*)]
    \item $G$ contains $W_5$ or
    \item $G$ contains a (not necessarily induced) spindle graph $M_4$ or $M_7$ or
    \item $G$ is 3-colourable.
\end{enumerate}
\end{corollary}

Moreover, for $P_6$-free graphs we can recall that in \cite{CGSZ} the following theorem was shown. 

\begin{theorem}[\rm \cite{CGSZ}]\label{CGSZ}
    A $P_6$-free graph is $3$-colourable if and only if it does not contain $F_1\cong K_4$, $F_2\cong W_5$, $F_3\cong M_7$, $F_4$,..., $F_{24}$ as a subgraph, defined in \cite{CGSZ}. 
\end{theorem}
It is easy to check that $F_2\cong W_5$, $F_3\cong M_7$, $F_{12}$, $ F_{15}$ and $F_{18}$ (see Figure~\ref{F123}) are the only $(K_4,bull)$-free graphs.
Note that $F_{18}$ is well known as the Mycielski graph. This leads to another corollary from Theorem \ref{thm:bullH}. 

\begin{figure}[htb]
    \begin{center}
    \begin{subfigure}{0.3\textwidth}
    \includegraphics[width=0.9\linewidth]{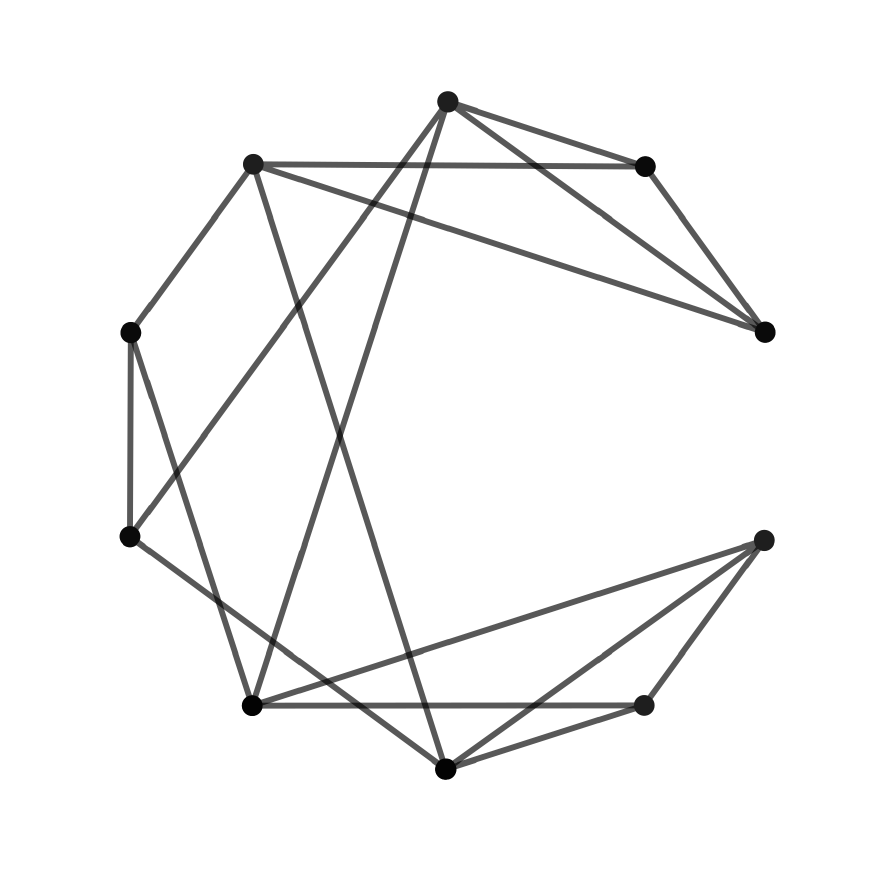} 
    \caption*{$F_{12}$}
    \end{subfigure}
    \begin{subfigure}{0.3\textwidth}
    \includegraphics[width=0.9\linewidth]{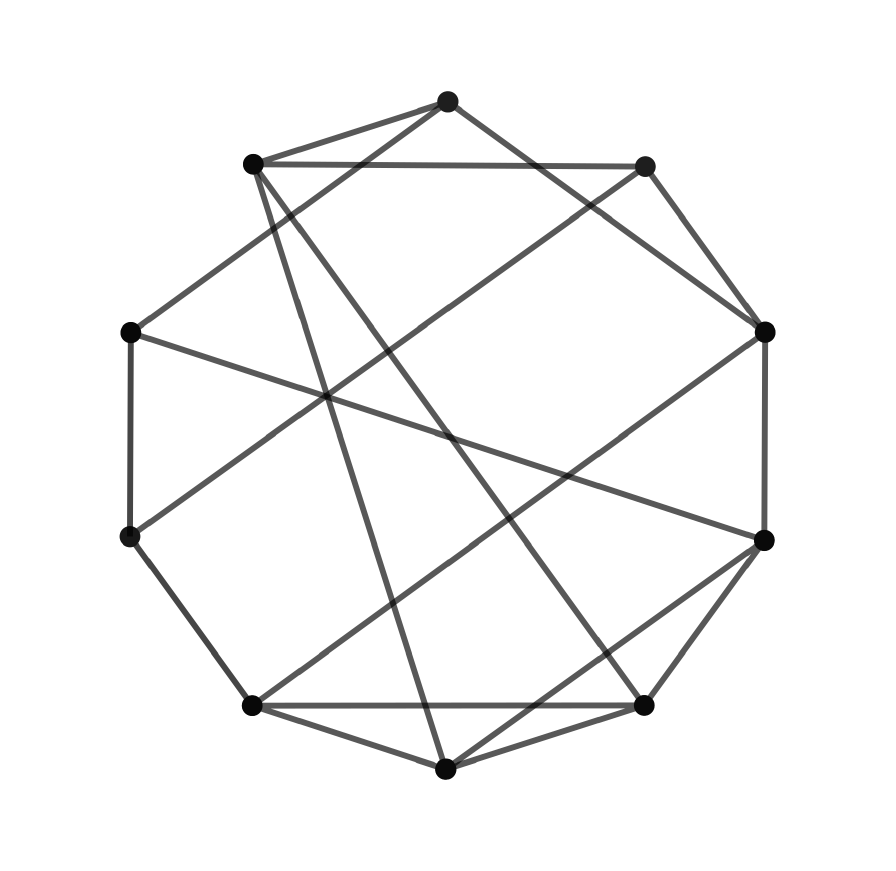}
    \caption*{$F_{15}$ }
    \end{subfigure}
    \begin{subfigure}{0.3\textwidth}
    \includegraphics[width=0.9\linewidth]{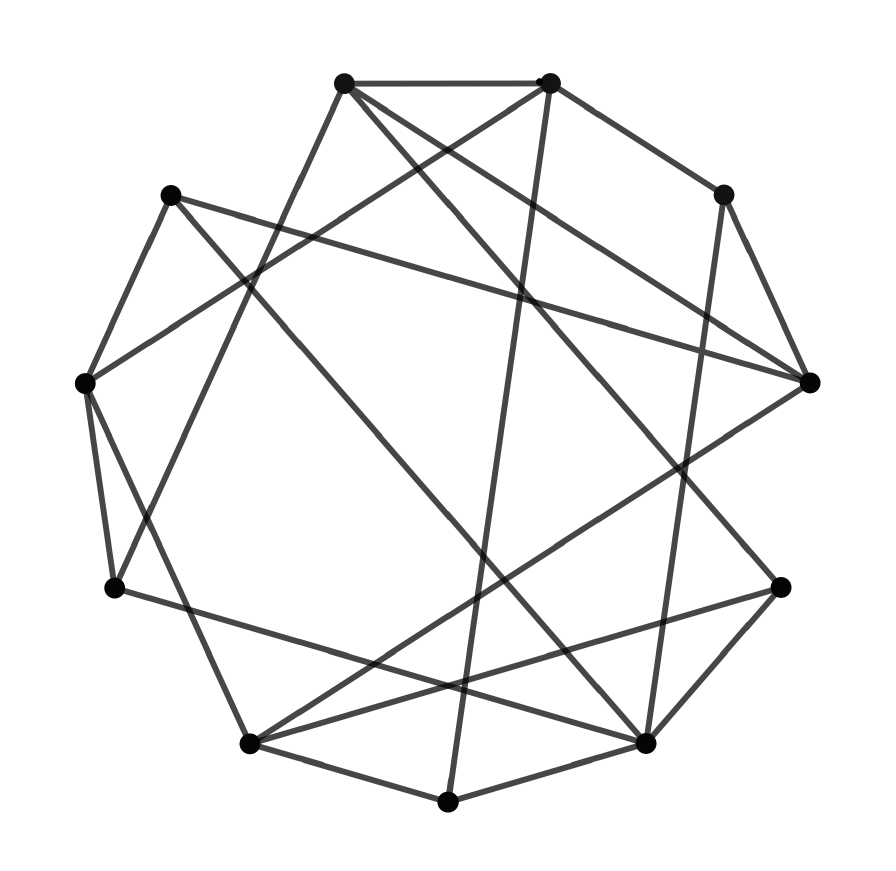}
    \caption*{$F_{18}$}
    \end{subfigure}
    \caption{The graphs $F_{12}$, $ F_{15}$ and $F_{18}$  from Theorem \ref{CGSZ}.}
    \label{F123}
    \end{center}
\end{figure}

\begin{corollary}
    Let $G$ be a $(bull, P_6)$-free graph. Then
    \begin{enumerate}[label=(\roman*)]
    \item $G$ contains $W_5$ or
    \item $G$ contains a (not necessarily induced) spindle graph $M_4$ or $M_7$ or
    \item $G$ contains $F_{12}$, $ F_{15}$  or $F_{18}$ or
    \item $G$ is 3-colourable.
\end{enumerate}
\end{corollary}

\vspace{1cm}
The organization of the paper is the following. In Section~\ref{sec:preliminary} we provide preliminary results and properties for $bull$-free graphs. Next, in Section~\ref{sec:proof_bull_chair} we prove Theorem~\ref{thm:bullchair} and in Section~\ref{sec:proof_bull_E} we prove Theorem~\ref{thm:bullE}. Finally, in Section~\ref{sec:algorithm} we show that the proofs of Theorem~\ref{thm:bullchair} and Theorem~\ref{sec:proof_bull_E} provide polynomial time certifying algorithms for $3$-colourability in the class of $(bull,H)$-free graphs for $H\in \{S_{1,1,2}, S_{1,2,2}\}$.

\section{Preliminary results}\label{sec:preliminary}


We recall that a {\it hole} in a graph $G$ is an induced cycle of
length at least $4$, and an {\it antihole} in $G$ is an induced subgraph whose
complement is a cycle of length at least~$4$.
A hole (antihole) is {\it odd} if it has an odd number of vertices.
As the main tool for proving Theorems~\ref{thm:bullchair} and \ref{thm:bullE},
we will use the well-known Strong Perfect Graph Theorem
shown by Chudnovsky et al.~\cite{ChRST06}.

\begin{theorem}[Chudnovsky et al.~\cite{ChRST06}]
\label{tSPGT}
A graph is perfect if and only if it contains neither an odd hole
nor an odd antihole as an induced subgraph.
\end{theorem}

In the following we will consider 3-colourability in subclasses of $bull$-free graphs. Here are some useful reductions:
\begin{itemize}
\item
If $\Delta(G) \leq 3,$ then $G$ is 3-colourable by Brook's Theorem.
\item
If $G$ has a vertex $w$ of degree at most 2, then $G$ is 3-colourable if and only if $G-w$ is 3-colourable. So we can reduce 
$G$ to $G-w.$
\item
If $G$ contains $K_4=M_4$, then $G$ is not 3-colourable.
\item If a graph $G$ contains an odd antihole $\overline{C_{2t+1}}$ with $t\geq 4$, then $G$ contains $K_4$. If $t=3$, then $G$ contains the spindle graph~$M_7$, and finally, if $t=2$, we have an antihole~$\overline{C_5}$, which is isomorphic to the hole~$C_5$. 
\item
If $G$ is not connected, then we can check 3-colourability for each component of $G$ seperately. Moreover, 
if $G$ has a cut-vertex $w,$ let $G_1, G_2, \ldots, G_t$ be the components of $G-w.$ Now we check whether each 
induced subgraph $G_i'=G[G_i \cup \{w\}]$ is 3-colourable. 
If all of $G'_1, \ldots, G'_t$ are 3-colourable, 
then we can combine their 3-colourings to obtain a 3-colouring of $G$.
\end{itemize}


These reductions show that we can restrict our 3-colourability test to the class of 
$bull$-free graphs that are 2-connected, $K_4$-free, and where $\delta(G) \geq 3.$
Furthemore, we can assume without losing generality that the graph $G$ contains an odd hole $C_{2p+1}$.

\subsection{Properties for $bull$-free graphs}

Let $Q=v_1 v_2 \ldots v_p v_1$ be the smallest induced odd hole in the graph $G$ and $w\in V(G) \setminus Q$. 
We define $q(w)$ as the largest $i$ such that $w$ has $i$ consecutive neighbours on the cycle~$Q$. Thus, there is $1\leq j \leq p$ satisfying $\{v_j, v_{j+1}, \ldots, v_{j+i-1}\}\subset N_{Q}(w)$. 
All indices are taken modulo $p$.

We will prove some useful facts about this value. 
       
\begin{fact}\label{values_q} 
    If $p>5$, then $q(w) \in \{1, 3\}$. 
    If $p=5$, then $q(w)\in \{1, 3, 4\}$. 
\end{fact}

\begin{proof} 
    Firstly, note that if $q(w)=2$, then the set of vertices $\{v_{j-1}, v_j, v_{j+1}, v_{j+2}, w\}$ induces $bull$. 
    If $4 \leq q(w) < p$ and $p \geq 7$, then the set of vertices $\{v_{j-1}, v_j, v_{j+1}, v_{j+3}, w\}$ induces $bull$. 
    If $q(w)=p$, then the graph $G$ contains an odd wheel $W_p$.
\end{proof}
        
\begin{fact}\label{q3}
    If $q(w)=3$, then $d_{Q}(w)=3$.
\end{fact}

\begin{proof}
    Suppose $q(w)=3$ and there is $v_k \in N_Q(w)$ with $v_k \notin \{v_j, v_{j+1}, v_{j+2}\}$. Hence, we have $p \geq 7$.
    Then one of the sets $\{v_{j-1}, v_j, v_{j+1}, v_k, w\}$ or $\{v_{j+1}, v_{j+2}, v_{j+3}, v_k, w\}$ induces $bull$.
\end{proof}
       
\begin{fact}\label{q1}
    If $q(w)=1$, then $d_{Q}(w)\in\{1, 2\}$. 
    Moreover, if $q(w)=1$ and $d_{Q}(w)=2$, then there is $i$ such that $N_{Q}(w)=\{v_i, v_{i+2}\}$.
\end{fact}

\begin{proof}
    Suppose $w$ has two neighbours $v_i, v_j$ in $Q$, satisfying $|i-j| \geq 3$, where $i>j$ . Then either the cycle $w v_i v_{i+1} \ldots v_j$ or the cycle $w v_j v_{j+1} \ldots v_i$ is odd. This cycle must contain an induced odd cycle $Q'$, which is shorter than $Q$. Since $w$ has no consecutive neighbours on $Q$, the cycle $Q'$ is not $K_3$.  
\end{proof}

\subsection{Properties for $(bull, E)$-free graphs}

We can now define the following sets:

\noindent $\bullet$ $A_i= \{ v \in V\setminus Q : N_Q(v)=\{v_i\} \}$.

\noindent $\bullet$ $B_i= \{ v \in V\setminus Q : N_Q(v)=\{v_i, v_{i+2}\} \}$.  

\noindent $\bullet$ $C_i= \{ v \in V\setminus Q : N_Q(v)=\{v_i, v_{i+1}, v_{i+2}\} \}$.

\noindent $\bullet$ $D_i= \{ v \in V\setminus Q : N_Q(v)=\{v_i, v_{i+1}, v_{i+2}, v_{i+3}\} \}$.

Let $A'_i= \{v\in A_i: \exists v'\in A_i, vv' \in E(G)\}$ and $B'_i= \{v\in B_i: \exists v'\in B_i \cup C_i, vv' \in E(G)\}$. Then let $A_i^*=A_i\setminus A_i'$ and $B_i^* = B_i\setminus B_i'$.

Let also $A = \bigcup_{i=1}^p A_i$, $B = \bigcup_{i=1}^p B_i$, $B' = \bigcup_{i=1}^p B_i'$,  $C = \bigcup_{i=1}^p C_i$ and $D = \bigcup_{i=1}^p D_i$. 

We want to prove that $V(G) = Q \cup A \cup B \cup C \cup D$. 
This is true due to the facts above and to the following lemma. 

\begin{Lemma}\label{lem:Qdominat}
    $Q$ is a dominating set in $G$.
\end{Lemma}

\begin{proof}
    Suppose there exists a vertex $v \in V(G)$ for which $\dist(v, Q)=2$. 
    Hence, there exist $v' \notin Q$ and $v_i \in Q$ such that $v'v_i, vv' \in E(G)$.  
    If $v'v_j \in E(G)$ for every $j$, then the set of vertices $\{v', v_1,\ldots, v_p\}$ induces odd wheel.
    If there exists~$k \in [p]$ such that $v'v_k, v'v_{k-1}\in E(G)$ and $ v'v_{k+1} \notin E(G)$, then the set $\{v, v', v_{k-1}, v_k, v_{k+1}\}$ induces $bull$.
           
    If none of these two cases occur, then $q(v')=1$. 
    Since $Q$ is an odd cycle, there exists $j \in [p]$ such that $v'v_{j-1}, v'v_{j+1}, v'v_{j+2}\notin E(G)$ and $ v'v_j \in E(G)$. 
    Then the set of vertices $\{v, v', v_j, v_{j+1}, v_{j+2}, v_{j-1}\}$ induces $E$ (so it also induces $chair$).
\end{proof}
       
\begin{fact}\label{mozkr}\label{f:B-B_noedge} 
    Let $w \in B_i \cup C_i$ and $w'\in B_j \cup C_j$. 
    If $|i-j|\geq 2$, then $ww' \notin E(G)$.
\end{fact}

\begin{proof}
    We can assume without losing generality that $2\leq j-i < p- (j-i)$. 
    Suppose $ww' \in E(G)$. 
    Let us consider possible cases. 

    If $j-i\geq 3$, then one of the sets $\{v_{i+2}, v_{i+3},
    \ldots ,v_{j}, w', w\}$ or $\{v_{j+2}, v_{j+3},\ldots, v_{i}, w, w'\}$ induces smaller odd cycle in $G$. 

    If $j-i=2$ and $p>5$, then the set of vertices $\{v_i, w, v_j, w', v_{j+2}\}$ induces $bull$. 

    If $j-i=2$, $p=5$ and $w' \in B$, then the set of vertices $\{v_i, w, w', v_j, v_{j+1}\}$ induces $bull$. 
    Analogously if $w \in B$. 

    If $j-i =2$, $p=5$ and $w, w' \in C$, then the graph $G$ contains the spindle graph~$M_7$.
\end{proof}

\begin{fact}\label{sasiedC}
    If $C_i \neq \emptyset$, then $C_{i+1}, C_{i-1} = \emptyset$.
\end{fact}

\begin{proof}
    Suppose that $w \in C_i$ and $w' \in C_{i+1}$. 
    If $ww' \in E(G)$, then the induced graph $G[\{w, v_{i+1}, v_{i+2}, w'\}]$ is complete. 
    If $ww' \notin E(G)$, then the set of vertices $\{v_{i-1}, v_i, w, v_{i+1},$ $ w'\}$ induces $bull$. 
\end{proof}
    
\begin{fact}\label{mozkrs}
    If $w \in B_i\cup C_i$, $w' \in B_{i+1} \cup C_{i+1}$ and $ww' \in E(G)$, then $w\in B^*_i$ or $w' \in B^*_{i+1}$.
\end{fact}

\begin{proof}
    Assume $ww'\in E(G)$. Let us consider possible cases. 
    
    If $w \in C_i$, then $w' \notin C_{i+1}$ by Fact~\ref{sasiedC}.
    
    If $w\in C_i$ and $w' \in B_{i+1}'$, then there exists $w''\in B_{i+1}'$ such that $w'w'' \in E(G)$ and either the set of vertices $\{v_{i-1}, v_i, w, v_{i+1}, w''\}$ induces $bull$ (if $ww'' \notin E(G)$) or the set~$\{w, v_{i+1}, w', w''\}$ induces $K_4$ (otherwise). 

    If $w \in B_i'$ and $w' \in B_{i+1}'$, then there exist $w'' \in B_i' \cup B_{i}$ and $w''' \in B_{i+1}' \cup C_{i+1}$ such that $ww'', w'w''' \in E(G)$. Then the set of vertices $\{v_{i-1}, v_i, w', w, w''\}$ induces $bull$ (if $w'w'' \notin E(G)$), or the set $\{v_{i+4}, v_{i+3}, w''', w', w\}$ induces $bull$ (if $ww''' \notin E(G)$), or the set  $\{v_{i+4}, v_{i+3}, w''', w', w''\}$ induces $bull$ (if $w'w'' \in E(G)$, but $w''w''' \notin E(G)$), or the set $\{w, w', w'', w'''\}$ induces $K_4$ (otherwise).
\end{proof}


\section{Proof of Theorem \ref{thm:bullchair}}\label{sec:proof_bull_chair}

Note that if $G$ is a $(bull, chair)$-free graph, then the sets $A$ and $B$ are empty. It is true, because if $q(w)=1$, then (since $Q$ is an odd cycle) there exists $i \in [p]$ such that $wv_{i-1}, wv_{i+1}, wv_{i+2} \notin E(G)$ and $wv_i \in E(G)$. Then the set of vertices $\{v_{i-1}, v_i, v_{i+1}, v_{i+2}, w\}$ induces $chair$. 

Moreover, we can point out that for every $i\in [p]$ we have $|C_i|\leq 1$. It is true, because if we have two distinct vertices $w, w' \in C_i$, then either the graph $G[\{v_i, v_{i+1}, w, w'\}]$ is complete (if $ww' \in E(G)$) or the set of vertices $\{v_{i-2}, v_{i-1}, v_i, w, w'\}$ induces $chair$ (if $ww' \notin E(G)$). 

Consider the case with $p=5$. We know that $V(G)=Q \cup C \cup D$. Our assumption is that $\delta(G)\geq 3$, so every vertex from $Q$ must have at least one neighbour in the set~$C\cup D$. Since $p=5$, the graph $G$ always contains $M_7$.

Now, we will describe the structure of the graph $G$ for $p>5$. By Fact \ref{values_q} we have $V(G) = C\cup Q$. Moreover, by Fact \ref{sasiedC}, $|N_Q(w) \cap N_Q(w')| \leq 1$ for every $w, w' \in C$. Then $|C| \leq \frac{p-1}{2}$. Let us recall that $C$ is an isolated set of vertices (by Fact \ref{mozkr}).

This graph is either $3$-colourable or contains the spindle graph. If $|C|=\frac{p-1}{2}$, then it is easy to see that $G$ is the spindle graph of order $\frac{3(p-1)}{2}+1$. Otherwise, there either exists vertex $v \in Q$ such that $d_C(v) = 0$ or there exist two vertices $v_i, v_j \in Q$ such that $C_i, C_{j-2} \neq \emptyset, C_{i-2}, C_{j} = \emptyset$ and $j=i+2k$, where $0<k< \frac{p-1}{2}$. The first case is impossible due to our assumption that $\delta(G)>2$. In second case we colour vertices $v_i, v_{i+1}, \ldots, v_j$ alternately with colours blue and red (starting with blue), and the rest of~$Q$ alternately with colours red and green. Finally, we colour vertices from $C$ with remaining colours. Note that for every vertex $w \in C$ we have at least one free colour - the colouring method provides us that for every $i$ such that $C_i\neq \emptyset$, vertices $v_i$ and $v_{i+2}$ get the same colour. Therefore, the obtained colouring is proper.


\section{Proof of Theorem \ref{thm:bullE}}\label{sec:proof_bull_E}

The proof of Theorem \ref{thm:bullE} will be split into two cases. 

\subsection{Case $p>5$}

Notice that in this case, if $w\in A_i$, then the set of vertices $\{v_{i-2}, v_{i-1}, v_i, v_{i+1}, v_{i+2}, w\}$ induces~$E$. Then (because $D=\emptyset$ due to Fact \ref{values_q}), we have $V(G)=Q \cup B \cup C$. 

Let us recall that due to Fact \ref{mozkr} edges $ww'$ non-incident to the cycle $Q$ can exist only for $w\in B_i \cup C_i$ and $w' \in B_i \cup C_i \cup B_{i+1} \cup C_{i+1}$.

To shorten our considerations, we will call an $ww'$ a ``1-type edge" if $w, w' \in B_i \cup C_i$, and a ``2-type edge" if $w \in B_i \cup C_i, w' \in B_{i+1} \cup C_{i+1}$. Of course, every edge non-incident to $Q$ is either 1-type or 2-type. Fact \ref{mozkrs} tells us that if $ww'$ is an 2-type edge, then $w \in B^*$ or $w' \in B^*$. It is obvious, that no 
1-type edge is incident to the set~$B^*$. 
Therefore, the graph 
$G' = G[V(G)\setminus (B^*)]$ does not contain any 2-type edge and does contain all 1-type edges.

Note that if there exists a proper $3$-colouring $c'$ of the graph $G'$, then there also exists a~proper colouring $c$ of the graph $G$. 

Assume that $c'$ is a proper $3$-colouring of the graph $G'$. Let us precolour with $c'$ all vertices outside the set $B^*$. By Fact \ref{mozkr}, every neighbour of non-precoloured vertex $w \in B_i^*$ is  $v_i$, or $v_{i+2}$, or belongs to the set $C_{i-1}\cup B_{i-1} \cup C_{i+1}\cup B_{i+1}$. That means every neighbour of $w$ is also incident to the vertex $v_{i+1}$. Thus, the vertex $w$ can get the colour $c'(v_{i+1})$.   

How can we decide whether a proper $3$-colouring of $G'$ exists or not?  We want to show that~$c'$ exists if and only if there exists a proper colouring $c''$ of the cycle $Q$ satisfying the following property:
\begin{equation}\label{prop}
    \forall i \in [p]:  C_i \cup B_i' \neq \emptyset \Rightarrow c''(v_i)=c''(v_{i+2}).
\end{equation}

Of course, if $C_i \cup B_i' \neq \emptyset$, then any proper colouring must assign the same colour to the vertices $v_i$ and $v_{i+2}$. The inverse implication is true due to the fact that the graph $G'$ does not contain any 
$2$-type edge and to the following observation.
       
\begin{fact}\label{bip}
    The graph $G[B_i \cup C_i]$ is bipartite for any $i$.
\end{fact}          
\begin{proof}
    Suppose $G[B_i \cup C_i]$ contains an odd cycle. Then it contains the induced odd cycle~$w_1w_2 \ldots w_s$, where $s\geq 3$, and the set of vertices $\{v_i, w_1, w_2, \ldots, w_s w_1\}$ induces either~$K_4$ or an odd wheel. 
\end{proof}  

Thus, having $c''$, we can construct $c'$ in a very simple way, assigning every vertex the first available colour. 

To find the colouring $c''$ satisfying the property (\ref{prop}), we proceed according to the following algorithm:
\begin{enumerate}
    \item Colour vertex $v_1$ with red. 
    \item Colour with red all those vertices whose colouring is enforced by the property (\ref{prop}).
    \item If there occurres a colour conflict, stop. The graph $G$ contains the spindle graph. 
    \item Let $k$ be an index such that $v_{k-2}$ is non-coloured and $v_k$ is red. Colour $v_{k-1}$ with green.
    \item Colour with green all those vertices whose colouring is enforced by the property (\ref{prop}).
    \item If there there occurred colour conflict, stop. The graph $G$ contains the spindle graph. 
    \item Colour vertex $v_{k-2}$ with blue. 
    \item Colour with blue every non-coloured vertex $v_{k-2l}$, where $l\in \N$.
    \item Colour with red all the remaining vertices. 
\end{enumerate}
\noindent Steps 8 and 9 are possible, since $Q$ is an odd cycle. Using this procedure we obtain a~proper colouring of the cycle $Q$. 

\subsection{Case $p=5$}

Since $G$ is $(bull,E)$-free graph, and $C_5$ is a dominating cycle (by Lemma \ref{lem:Qdominat}), it follows that $V(G) = Q \cup A \cup B \cup C \cup D$.
Assuming $G$ does not contain $W_5$, $K_4$ and $M_7$, we will prove a number of properties of these subsets.

\begin{fact}\label{f:p5_bipartite}
    The graphs $G[C_i \cup B_i]$ and $G[A_i]$ are bipartite for any $i$. 
\end{fact}

\begin{proof}
   The proof is analogous to the proof of Fact \ref{bip}.
\end{proof}

\begin{fact}\label{f:A-AB_edges}
    If $w \in A_i$ and $w'\in A_{i+1}\cup A_{i-1} \cup B_{i-1}$, then $ww' \in E(G)$. 
\end{fact}
\begin{proof}
    Suppose $w' \in A_{i+1}\cup B_{i-1}$ and $ww' \notin E(G)$. Then the set of vertices $\{w, v_i, v_{i+1},$ $v_{i+2},$ $v_{i+3},$ $ w'\}$ induces~$E$. Analogously for $A_{i-1}$. 
\end{proof}

\begin{fact}\label{f:A-ABC_noedge}
    If $w\in A_i$ and $w'\in A_{i+2}\cup A_{i+3}\cup B_i\cup B_{i+3}\cup C_i \cup C_{i+1} \cup C_{i+2} \cup C_{i+3}$, then $ww' \notin E(G)$. 
 \end{fact}

\begin{proof}
    Suppose $ww' \in E(G)$.
    
    If $w' \in A_{i+2}$, then the set of vertices $\{v_{i+3}, v_{i+4}, v_i,$ $v_{i+1}, w, w'\}$ induces~$E$. Analogously for $A_{i+3}$. 

    If $w' \in B_i \cup C_i$, then the set of vertices $\{v_{i-1}, v_i, w, w', v_{i+2}\}$ induces~$bull$. Analogously for $B_{i+3} \cup C_{i+3}$.

    If $w' \in C_{i+1}$, then the set of vertices $\{ w, w', v_{i+2}, v_{i+3}, v_{i+4} \}$ induces~$bull$. Analogously for $C_{i+2}$.
\end{proof}

\begin{fact}\label{f:A'B'_max2}
    There are at most two $i,j$ such that $A_i'$ and $A_j'$ are nonempty. Moreover $|i-j|>1$.
\end{fact}

\begin{proof}
    Suppose $w, w' \in A_i'$ and $u,u' \in A_{i+1}'$, where $ww', uu' \in E(G)$. Then by Fact~\ref{f:A-AB_edges} the set of vertices $\{w, w', u, u'\}$ induces $K_4$. Thus $|j-i|>1$. Since $Q$ consists of five vertices, the conclusion holds.
\end{proof}

\begin{fact}\label{f:A-AB_edge_or_no}
    If $w\in A_i$, $w' \in A_{i+1}$ and $w'' \in B_{i+2}$, then $ww'' \notin E(G)$ or $w'w'' \notin E(G)$.
    
\end{fact}
\begin{proof}
    Suppose $ww'', w'w'' \in E(G)$. By Fact \ref{f:A-AB_edges} $ww' \in E(G)$. Then  the set of vertices $\{v_i, w, w', w'', v_{i+2}\}$ induces $bull$.
\end{proof}

\begin{fact}\label{f:C-AB_edges}
    If $w\in C_i$ and $w' \in A_{i+1} \cup B_{i-1}\cup B_{i+1}$, then $ww' \in E(G)$.
\end{fact}

\begin{proof}
    Suppose $ww' \notin E(G)$.
    If $w' \in A_{i+1} \cup B_{i+1}$, then the set of vertices $\{ v_{i-1}, v_{i}, v_{i+1}, w, w' \}$ induces~$bull$. Analogously for $w' \in B_{i-1}$
\end{proof}

\begin{fact}\label{f:C-A_noedge}
    If $w \in C_i$ and $w' \in A_i\cup A_{i+2} \cup A_{i+3} \cup A_{i+4} $, then $ww' \notin E(G)$. 
\end{fact}

\begin{proof}
    Suppose $ww' \in E(G)$. 

    If $w' \in A_i$, then the set of vertices $\{v_{i+3}, v_{i+2}, v_{i+1}, w, w'\}$ induces $bull$. Analogously for $w' \in A_{i+2}$ 
    
    If $w' \in A_{i+3}$, then the set of vertices $\{v_{i-1}, v_i, v_{i+1}, w, w'\}$ induces $bull$. Analogously for $w' \in A_{i+4}$.

\end{proof}

\begin{fact}\label{f:ifone_thentwo}
    If $w,w' \in A_i'$, $ww' \in E(G)$ and there exists $w'' \in V(G)\setminus (A_i \cup D)$ such that $ww'' \in E(G)$, then $w'w'' \in E(G)$.
\end{fact}

\begin{proof}
    Suppose $ww', ww'' \in E(G)$ and $w'w'' \notin E(G)$. 

    By Fact \ref{f:A-ABC_noedge} $w'' \notin A_{i+2}\cup A_{i+3}\cup B_i\cup B_{i+3}\cup C_i \cup C_{i+1} \cup C_{i+2} \cup C_{i+3}$.   Let us consider possible cases. 
    
    If $w'' \in A_{i+1} \cup B_{i+1}$, then the set of vertices $\{v_{i-1},$ $v_i,$ $w,$ $w',$ $w''\}$ induces $bull$. Analogously for $A_{i-1}$ and $B_{i+2}$.

    If $w'' \in B_{i-1}$, then by Fact \ref{f:A-AB_edges} we have $w'w'' \in E(G)$.

    If $w'' \in C_{i-1}$, then by Fact \ref{f:C-AB_edges} the set of vertices $\{v_i, w, w', w''\}$ induces $K_4$.
    
\end{proof}

\begin{fact}\label{f:C}
    There is only one $i$ such that $C_i \neq \emptyset$. Moreover, if $C_i \neq \emptyset$ then $A_{i+1}' \cup B_{i+1}' \cup B_{i+2}' \cup B_{i+3}' \cup B_{i+4}' = \emptyset$. 
\end{fact}
\begin{proof}
    Let $w \in C_i$. We consider possible cases.
    
    If $w', w'' \in A_{i+1}'$ and $ww' \in E(G)$, then by Fact \ref{f:C-AB_edges} the set of vertices $\{v_{i+1}, w, w', w''\}$ induces $K_4$. 
    
    If $w' \in C_{i+1}$, then either the set of vertices $\{v_{i-1}, v_i, w, v_{i+1}, w'\}$ induces~$bull$ (if $ww' \notin E(G)$) or the set of vertices $\{v_i, v_{i+1}, w, w'\}$ induces $K_4$ (otherwise). Analogously for $w' \in C_{i-1}$.

    If $w' \in C_{i+2}$, then $G$ contains the spindle graph $M_7$. Analogously for $w' \in C_{i+3}$.

    If $w', w'' \in B_{i+1}'$, then the set of vertices $\{v_{i-1}, v_i, w, v_{i+1}, w'\}$ induces $bull$ (if $ww' \notin E(G)$), or the set of vertices $\{v_{i+4}, v_{i+3}, w'', w', w \}$ induces $bull$ (if $ww' \in E(G)$ but $ww'' \notin E(G)$), or the set of vertices $\{v_i, w, w', w''\}$ induces $K_4$. Analogously for $w',w'' \in B_{i-1}$.

    If $w', w'' \in B_{i+2}'$, then $G$ contains the spindle graph $M_7$. Analogously for $w',w'' \in B'_{i+3}$.
\end{proof}

\begin{fact}\label{f:ifone_thentwo2} 
    If $w, w' \in B_i'$, $ww' \in E(G)$ and there exists $w'' \in V(G) \setminus (B_i \cup C_i \cup D)$ such that $ww'' \in E(G)$, then $w'w'' \in E(G)$.  
\end{fact}

\begin{proof} 
    By Facts \ref{f:B-B_noedge}, \ref{f:A-ABC_noedge} we have $w'' \notin B_{i+2} \cup B_{i+3} \cup A_{i+2} \cup A_i$.
    
    By Fact \ref{f:C} we know that $C\setminus C_i = \emptyset$. Then $w'' \notin C$. If $w'' \in A_{i+1}$, then $w'w''$ exists by Fact \ref{f:A-AB_edges}.

    Suppose now $w'' \in B_{i+1} \cup A_{i+3}$ and $ww'' \in E(G), w'w'' \notin E(G)$. Then the set of vertices $\{v_{i-1}, v_i, w', w, w''\}$ induces $bull$. Analogously if $w'' \in B_{i-1} \cup A_{i+4}$.

\end{proof}

\begin{fact}\label{f:D-AB_edges}
    Let $w \in D_i$.
    If $w' \in A_{i+1}\cup A_{i+2} \cup B_{i-1}\cup B_i \cup B_{i+1} \cup B_{i+2}$, then $ww' \in E(G)$. 
\end{fact}

\begin{proof}
    Suppose $ww' \notin E(G)$. 
    
    If $w' \in A_{i+1}$, then the set of vertices $\{ v_{i-1}, v_{i}, w, v_{i+1}, w'\}$ induces $bull$. Analogously for $w' \in A_{i+2}$. 
        
    If $w' \in B_i \cup B_{i+1}$, then the set of vertices $\{w', v_i, v_{i+1}, w, v_{i+4}\}$ induces $bull$. 

    If $w' \in B_{i+2}$, then the set of vertices $\{v_i, w, v_{i+3}, v_{i+2}, w'\}$ induces $bull$. 
\end{proof}

\begin{fact}\label{f:D}
    Let $D \neq \emptyset$. Then $D=D^*_i$ for some $i$ and $A_i \cup A_{i+3} \cup A'_{i+1} \cup A'_{i+2} \cup B' \cup C= \emptyset$.
\end{fact}

\begin{proof}
    Suppose $w \in D_i$. 
    If $w'\in D\setminus D_i$, then $G$ contains the spindle graph $M_7$.
    
    If $w' \in D_i$ and $ww' \in E(G)$, then the set of vertices $\{ v_i, v_{i+1}, w, w' \}$ induces $K_4$.

    If $w' \in A_i$, then either the set of vertices $\{w', w, v_i, v_{i+1}, v_{i+3}\}$ induces $bull$ (if $ww' \notin E(G)$) or the set of vertices $\{v_{i-1}, v_i, w', w, v_{i+2}\}$ induces $bull$ (otherwise). Analogously for $A_{i+3}$.

    Suppose $w', w'' \in A'_{i+1}$ and $w'w'' \in E(G)$. By Fact~\ref{f:D-AB_edges}  the set of vertices $\{w, w', w'', v_{i+1} \}$ induces $K_4$.
    Analogously for $A'_{i+2}$.

    If $w'\in C_i$, then the set of vertices $\{v_{i+4}, v_{i+3}, w, v_{i+2}, w'\}$ induces~$bull$ if  and $ww' \notin E(G)$ or the set of vertices $\{v_i, v_{i+1}, w, w'\}$ induces $K_4$ if $ww' \in E(G)$. Analogously for $w'\in C_{i+1}$.

    If $w' \in C_{i+2} \cup C_{i+3} \cup C_{i+4}$, then the graph $G$ contains the spindle graph $M_7$.

    Finally, suppose $w'w'' \in E(G)$. 

    If $w', w'' \in B_i'$, then by Fact \ref{f:D-AB_edges} the set of vertices $\{v_i, w, w', w''\}$ induces $K_4$. Analogously for $w', w'' \in B_{i+1}'$.

    If $w', w'' \in B_{i+2}' \cup B_{i+3}' \cup B_{i+4}'$, then the graph $G$ contains the spindle graph~$M_7$.
\end{proof}

Due to all the facts above, we can distinguish following types of possible edges:
\begin{itemize}
    \item Type $0$: edges incident to the cycle $Q$.
    \item Type $1$: edges $ww'$ such that $w, w' \in B_i \cup C_i$.
    \item Type $2$: edges $ww'$ such that $w\in B_i\cup C_i$ and $w'\in B_{i+1} \cup C_{i+1}$.
    \item Type $3$: edges $ww'$ such that $w\in A_i$ and $w' \in B_{i-1} \cup C_{i-1}$.
    \item Type $4$: edges $ww'$ such that $w, w' \in A_i$.
    \item Type $5$: edges $ww'$ such that $w\in A_i$ and $w'\in A_{i+1}$.
    \item Type $6$: edges $ww'$ such that $w \in A_i$ and $w' \in B_{i+1}\cup B_{i+2}$. 
    \item Type $7$: edges non-incident to $Q$ and incident to the set $D$.
\end{itemize}

Let us recall that by Fact \ref{f:A-AB_edges} edges of types $3$ and $5$ are obligatory, that is, if respective sets are nonempty, every edge between them exists.

\begin{figure}[htb]
    \centering
    
    \begin{subfigure}{0.4\textwidth}
    \includegraphics[width=0.75\linewidth]{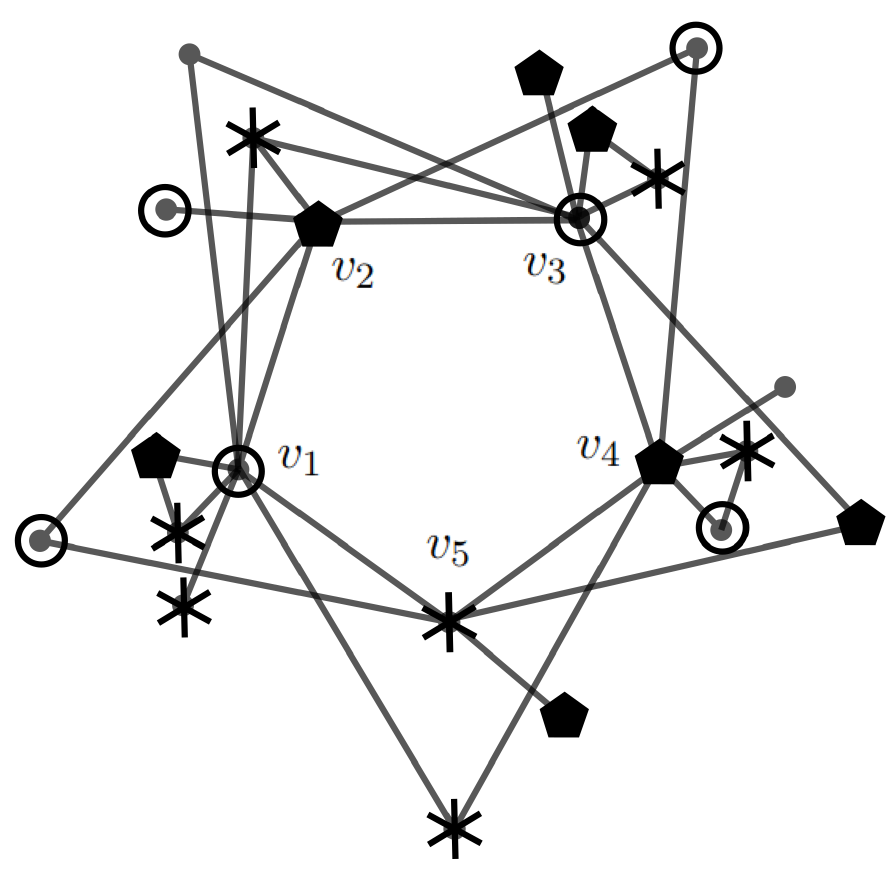} 
    \caption{Case 1/2.}
    
    \end{subfigure}
    \begin{subfigure}{0.4\textwidth}
    \includegraphics[width=0.75\linewidth]{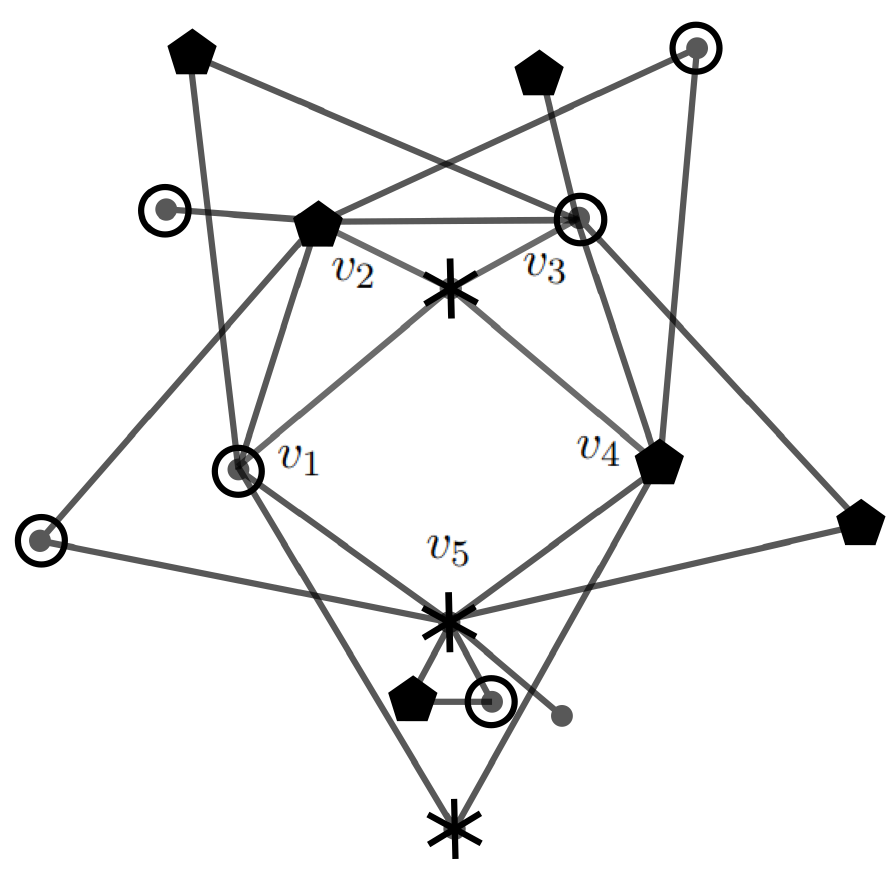}
    \caption{Case 3.}
    
    \end{subfigure}
    
\end{figure}

\textbf{Algorithm.}
    
    Let us colour the graph $G$ as follows.
    \begin{enumerate}
        \item We colour with red vertices $v_1$ and $v_3$, with green vertices $v_2$ and $v_4$, and with blue vertex $v_3$. 
        \item For $i \neq 1$ and $w\in B^*_i$ we colour $c(w)=c(v_{i+1})$.
        \item For $i\neq 4, 5$ and $w\in A^*_i$ we colour $c(w)=c(v_{i-1})$.
        \item We colour $C_1$ with blue and $A_1', A_3', A_4', B_1'$ with remaining colours. 
        \item If $D= \emptyset$, we colour $A_5^*$ with green. Then we colour with red every $w\in A_4^*$  adjacent to $A_3' \cup B_1'$. The rest of $A_4$ we colour with blue. We colour with blue every $w \in B_1^*$ adjacent to $A_5^*$. The rest of $B_1^*$ we colour with green. 
        \item If $D \neq \emptyset$, we colour $D_1$ with blue. Then we colour $B_1^*$ with green and $A_5'$ with red and green. Every $w\in A_5^*$ adjacent to $B_1^*$ we colour with red. The rest of $A_5^*$ we colour with green. 
    \end{enumerate}

\begin{proposition}
    The algorithm colours all vertices of $G$ and this colouring is proper. 
\end{proposition}

\begin{proof}

We want to prove that our algorithm colours all vertices of $G$ and this gives a proper colouring. Of course, it is easy to see that there will not occur any colour conflict on edges of types $0$, $2$ and $3$. By Fact \ref{f:p5_bipartite} there is no colour conflict also on edges of types $1$ and $4$ (we have two free colours, and the respective subgraphs are bipartite). The further proof will be split into three subcases. 

\textit{Case 1}. $C \neq \emptyset$. 

 Of course in this case $D= \emptyset$ (by Fact \ref{f:D}). Note that we can assume without loss of generality  that $C=C_1$, $B' = B_1'$ (by Fact \ref{f:C}) and  $A_5' = \emptyset$ (by Fact \ref{f:A'B'_max2}). 
 Let us also recall that in this case $A_2' = \emptyset$ (by Fact \ref{f:C}). 
 
    We start with checking the possible colour conflicts on the edges of type $5$. Of course, one of the sets $A_3', A_4'$ is empty (by Fact \ref{f:A'B'_max2}), so there is no colour conflict on the edges of type $5$ between sets $A_3$ and $A_4$. Suppose $A_1'$ and $A_5^*$ are both nonempty. 
    If $A_4' \neq \emptyset$, then the graph $G[\{v_4\} \cup A_4' \cup A_5^* \cup A_1' \cup \{v_1, v_2, v_3\} \cup C_1]$ contains $M_{10}$. If $A_3' \neq \emptyset $ and $A_4^* \neq \emptyset$ then the graph $G[A_5^* \cup A_1' \cup \{v_1, v_2, v_3\} \cup C_1 \cup A_3' \cup A_4]$ contains $M_{10}$. 
    Thus, at least one of the sets $A_3'$, $A_4^*$ is empty.
    Hence, we can consider the symmetry of the graph such that it swaps $v_4$ and $v_5$. Then one of the sets $A_1'$ or $A_5^*$ is empty. 

 Now we want to exclude colour conflicts on the edges of type $6$. If $A_1'$ is adjacent to $B_3^*$, then the graph $G[B_3^* \cup A_1' \cup \{v_1, v_2, v_3\} \cup C]$ contains $M_7$ (by Facts \ref{f:C-AB_edges}, \ref{f:ifone_thentwo}). Analogously for $A_3'$ incident to $B_4^*$. 
 If $A_4'$ is incident to $B_5^*$, then the graph $G[\{v_4\} \cup A_4' \cup B_5^* \cup \{v_2, v_3\} \cup C]$ contains $M_7$ (by Fact \ref{f:C-AB_edges}, \ref{f:ifone_thentwo}). 
 The colour conflict between $A_4'$ and $B_1^*$ may occur only if $B_1$ is blue. By definition of the colouring, $B_1$ is blue only if it is incident to $A_5^*$. By Fact~\ref{f:A-AB_edge_or_no} this is impossible. 

 If there is a colour conflict between $A_4^*$ and $B_5^*$, then there must be two adjacent vertices $w \in A_4^*$ and $w' \in B_5^*$, both coloured with red. But $w$ can be red only when it is adjacent to $B_1'$ (by Fact \ref{f:C-AB_edges} the graph $G[\{v_1, v_2\} \cup C \cup B_1' \cup \{w, w'\}]$ contains $M_7$) or when $A_3'$ is nonempty (and the graph $G[{w, w', v_2, v_3} \cup C \cup A_3']$ contains $M_7$). 

  The colour conflict between $A_4^*$ and $B_1^*$ can occur only if there are two vertices $w \in A_5^*$ and $w' \in B_1$, both coloured with blue. But if $w'$ is coloured with blue, by definition of the colouring, there is $w'' \in A_5^*$ such that $w'w'' \in E(G)$. It contradicts Fact \ref{f:A-AB_edge_or_no}, so this colour conflict is also impossible. 
 
 If $w \in A_5^*$ is adjacent to $w', w'' \in B_1'$, then $u \in A_4$ and $w'$ cannot be adjacent (otherwise the graph $G[\{w, w', w'', u\}]$ contains $K_4$ or $G[\{w, w', w'', v_1\} \cup A_4]$ contains $M_7$, by Fact \ref{f:ifone_thentwo2}). Thus, we can consider the symmetry of the graph such that it swaps $v_4$ and $v_5$. Note that this operation does not change our previous assumption. Indeed, before the reflection, if $A_4'\neq \emptyset$, then the graph $G[\{v_4\} \cup A_4' \cup A_5 \cup B_1' \cup \{v_3\}]$ contains $M_7$, and if $A_3' \neq \emptyset$ and $A_4^* \neq \emptyset$, then the graph $G[A_5^* \cup B_1' \cup \{v_3\} \cup A_3' \cup A_4^*]$ contains $M_7$.  

  By definition of the colouring, there is no colour conflict between $A_5^*$ and $B_1^*$.

\textit{Case 2}. $C \cup D = \emptyset$

In case $C \cup D = \emptyset$ we can also assume $B' = \emptyset$ (because otherwise we can find another $5$-cycle $Q'$ and a vertex incident to three consecutive vertices of $Q'$). By Fact \ref{f:A'B'_max2} we can assume that $A_2' \cup A_5' = \emptyset$. Also, we can assume that there is no $6$-type edges with vertices from $A'$ (again, by Fact \ref{f:ifone_thentwo}, otherwise we could find there another $5$-cycle so that we return to \textit{Case 1}) and that if $A'$ is nonempty, then $A_5$ is empty (because for $A_i' \neq \emptyset$ we have one of sets $A_{i+1}, A_{i+2}, A_{i+3}, A_{i+4}$ empty by the same argument as before). 

Now, let us check the possible colour conflicts. It is easy to see that if $A' = \emptyset$, our colouring is proper. Assume $A' \neq \emptyset$. Then, thanks to $A_5 = \emptyset$, there is no conflicts on edges of type $5$. The only possible (given our assumptions) conflicts may occur on edges of type $6$ with vertices from $A_4^*$. As in previous case, there is no conflict between $A_4^*$ and $B_1$. If $A_4^*$ is connected to $B_5^*$ and coloured with red, then by definition of our colouring $A_3' \neq \emptyset$. Hence, we can find another $5$-cycle so that we return to \textit{Case 1}. 

\textit{Case 3}. $D \neq \emptyset$

We will see that there cannot occur any colour conflict on edges of type $7$.
Let us recall that by Fact \ref{f:D} in this case we have $A_i \cup A_{i+3} \cup A'_{i+1} \cup A'_{i+2} \cup B' \cup C= \emptyset$. If $B_4$ is adjacent to $D=D_1$, then the graph $G$ contains $W_5$. 

Hence colour conflicts can occur only on the edges of type $6$ with vertices from $A_5$. If $A_5'$ is adjacent to $B_1^*$, then the graph $G[\{v_5\}\cup A_5' \cup B_1^* \cup \{v_3, v_4\} \cup D]$ contains $M_7$ by Fact~\ref{f:D-AB_edges}. Analogously for $A_5'$ adjacent to $B_2^*$. Now suppose there are $w \in A_5^*$,  $w' \in B_1^*$ and $w'' \in B_2^*$ such that $ww', ww'' \in E(G)$. Note that $w'w'' \in E(G)$ (because otherwise the set $\{v_5, v_1, w', u, w''\}$ induces $bull$, for any $u\in D$). Then the set of vertices $\{v_1, w', w, w'', v_4\}$ induces $bull$.

\end{proof}

\subsection{Proof of Theorem~\ref{thm:bullH}}
We can follow the proof of Theorem~\ref{thm:bullE} with a few changes. 
Because there are no 5-cycles, we only have to consider the case $p > 5$. 
Notice that in this case, if $w\in A_i$, then the set of vertices $\{v_{i-3}, v_{i-2}, v_i, v_{i+1}, v_{i+2}, v_{i+3}, w\}$ induces~$S_{1,2,3}$. 
Then (because $D=\emptyset$ due to Fact \ref{values_q}), we have $V(G)=Q \cup B \cup C$. 
Now, we can follow the proof of Theorem~\ref{thm:bullE}.

\section{Certifying algorithms}\label{sec:algorithm}

\begin{theorem}
    There exists a polynomial time certifying algorithm for 3-colourability in the class of $(bull, H)$-free graphs for $H\in \{S_{1,1,2}, S_{1,2,2}\}$, and in the class of $(bull, C_5, H)$-free graphs with $H\in \{S_{1,1,3}, S_{1,2,3}\}$. 
\end{theorem}

\begin{proof} An odd hole can be found in polynomial time $O(n^9)$ by an algorithm in~\cite{ChSSS20}. So let $Q$ be this odd hole of length $p$. To check whether $G$ contains an odd wheel $W_{2t+1}$ for some $t\geq1$, one can check for every vertex $w \in V(G)$ in polynomial time $O(|E|)$ whether $G[N(w)]$ is bipartite. If this is not the case, then $G$ contains an odd wheel with center vertex $w$. In the case $p > 5$, all structural investigations can be performed in polynomial time and the algorithm either finds a proper 3-colouring or detects a spindle graph $M_{3t+1}$ for some $t\geq 3$. In the proofs for the case $p=5$, proper 3-colourings of $G$ or a subgraph from $\{M_4, M_7, M_{10}\}$ will be found in polynomial time.

Summarizing, all structural investigations and algorithms run in polynomial time and either find a proper 3-colouring of $G$ or detect an odd wheel or a spindle graph $M_{3i+1}$ for some $i \geq 1$.
\end{proof} 


\end{document}